\documentclass{amsart}
\usepackage{graphicx}
\usepackage{amsmath}
\usepackage{amsthm}
\usepackage{amssymb}%
\usepackage[numbers,square]{natbib}

\newcommand{\E}{\mathbb E}
\newcommand{\R}{\mathbb{R}}
\newcommand{\N}{\mathbb{N}}

\newcommand{\Z}{\mathbb{Z}}

\renewcommand{\P}{\mathbb{P}}
\newcommand{\Var}{\mathop{\mathrm{Var}}\nolimits}

\newcommand{\NN}{\mathcal{N}}
\newcommand{\FF}{\mathcal{F}}

\newcommand{\eps}{\varepsilon}

\newcommand{\todistr}{\stackrel{\mathrm{w}}{\to}}

\theoremstyle{plain}
\newtheorem{theorem}{Theorem}[]
\newtheorem{lemma}{Lemma}[]
\newtheorem{corollary}{Corollary}[]

\theoremstyle{definition}
\newtheorem{example}{Example}[]
\newtheorem{remark}{Remark}[]
\newtheorem{ass}{Assumption}[]

\theoremstyle{remark}

\def\crl#1{\crcr\noalign{\unpenalty\penalty 10000
\nointerlineskip \vbox to 0pt {
\dimen0=\lineskip \vskip \dimen0 minus 1000pt \hbox to \displaywidth{%
\hfil \refstepcounter{equation}\label{#1}(\theequation)}
\vskip 0pt minus 1000pt} \penalty 10000}}

\begin{document}

\title[Sums of independent products]{Limit laws for sums of independent random products: the lattice case}
\author{Zakhar Kabluchko}
\address{Institute of Stochastics, Ulm University, Helmholtzstr.\ 18, 89069 Ulm, Germany}
\begin{abstract}
Let $\{V_{i,j}; (i,j)\in\N^2\}$ be a two-dimensional array of i.i.d.\ random variables. The limit laws of the sum of independent random products
$$
Z_n=\sum_{i=1}^{N_n} \prod_{j=1}^{n} e^{V_{i,j}}
$$
as $n,N_n\to\infty$ have been investigated by a number of authors.
Depending on the growth rate of $N_n$, the random variable $Z_n$ obeys a central limit theorem, or has limiting $\alpha$-stable distribution.
The latter result is true for non-lattice $V_{i,j}$ only. Our aim is to study the lattice case. We prove that although the (suitably normalized) sequence $Z_n$ fails to converge in distribution, it is relatively compact in the weak topology, and describe its cluster set. This set is a topological circle consisting of semi-stable distributions.
\end{abstract}
\maketitle

\section{Introduction and statement of results}\label{sec:intro}
Let $\{V_{i,j}; (i,j)\in\N^2\}$ be a two-dimensional array of independent copies of a real-valued random variable $V$. Our main object of interest is the \textit{sum of independent random products}
\begin{equation}\label{eq:def_Zn}
Z_n=\sum_{i=1}^{N_n} \prod_{j=1}^{n} e^{V_{i,j}}.
\end{equation}
Here, $N_n$ is a sequence of positive integers converging to $\infty$. The limit laws of the random variable $Z_n$ as $n, N_n\to\infty$ have been studied by~\citet{bovier_etal02} for Gaussian $V$ (see Theorems~1.5, 1.6 therein), and by~\citet{cranston_molchanov05} for arbitrary $V$ with finite exponential moments. The study of $Z_n$ is motivated by a number of models in statistical physics. To mention only one example, if $V_{i,j}$ are Gaussian variables, then $Z_n$ is the partition function of the random energy model; see~\cite{bovier_etal02}.  The character of the limiting distribution of $Z_n$ depends on the growth rate of the sequence $N_n$. If the sequence $N_n$ grows fast in the sense that
\begin{equation}\label{eq:asympt_normal_conv}
\liminf_{n\to\infty} \frac1 n\log N_n >\lambda_2
\end{equation}
for some critical value $\lambda_2>0$ depending only on the distribution of $V$, then the random variable $Z_n$ obeys a central limit theorem with the usual normalization:
\begin{equation}\label{eq:conv_normal}
\frac{Z_n-\E Z_n}{\sqrt{\Var Z_n}}
\todistr \NN(0, 1), \;\;\; n\to\infty.
\end{equation}
If the sequence $N_n$ grows slowly in the sense that
\begin{equation}\label{eq:asympt_N}
\lambda:=\lim_{n\to\infty}\frac1n\log N_n\in (0,\lambda_2),
\end{equation}
then the central limit theorem breaks down. Instead, for suitable normalizing sequences $A_n$, $B_n$, we have
\begin{equation}\label{eq:conv_stable}
\frac{Z_n-A_n}{B_n}\todistr \FF_{\alpha}, \;\;\; n\to\infty,
\end{equation}
where $\FF_{\alpha}$ is an $\alpha$-stable distribution totally skewed to the right, and the stability parameter $\alpha\in (0,2)$ depends on $\lambda$. The proofs of~\eqref{eq:conv_normal} and~\eqref{eq:conv_stable} can be found in~\cite{cranston_molchanov05}. Unaware of~\cite{cranston_molchanov05}, the author proved essentially the same results in~\cite{kabluchko09b}. A functional version of these results can be found in~\cite{kabluchko09d}. There is also a transition between the two regimes~\eqref{eq:asympt_normal_conv} and~\eqref{eq:asympt_N} taking place at $\log N_n\approx \lambda_2 n$; see~\cite[Thm.~1.5(ii)]{bovier_etal02}, \cite[Thm.~1.2]{cranston_molchanov05}, as well as~\cite[Thm.~1.3]{kabluchko09d}.

In their proof of the stable limit law~\eqref{eq:conv_stable}, \citet{cranston_molchanov05} relied on an asymptotic expansion in the central limit theorem; see Theorem~1 on page~210 in~\cite{gnedenko_book54}. It has been overlooked in~\cite{cranston_molchanov05} that this result is true for \textit{non-lattice} distributions only; see page~212 in~\cite{gnedenko_book54} for a discussion of this fact. Recall that a random variable $V$ is called \textit{lattice} if there exist $h,a\in\R$ such that the values of $V$ are a.s.\ of the form $hn+a$, $n\in\Z$.

Our aim  is to investigate the lattice case. On the one hand, we will see that in this case the convergence to an $\alpha$-stable law breaks down. More precisely, there is no affine normalization which makes the sequence of random variables $Z_n$ weakly convergent. On the other hand, we will prove that for suitable $A_n$ and $B_n$, the sequence of random variables $(Z_n-A_n)/B_n$ is relatively compact in the weak topology and describe the set of weak accumulation points for this sequence. This set is a topological circle consisting of semi-stable distributions.


Let us state our results precisely.  Let $V$ be a non-degenerate random variable satisfying the Cram\'er condition
\begin{equation}\label{eq:def_varphi}
\psi(t):=\log \E e^{tV}<+\infty \text{ for all } t\in \R.
\end{equation}
Let  $I:\R\to [0,+\infty]$  be the Legendre--Fenchel transform of $\psi$ given by
\begin{equation}\label{eq:def_I}
I(\beta):=\sup_{t\in \R} (\beta t-\psi(t)), \;\;\; \beta\in\R.
\end{equation}
Following~\cite{cranston_molchanov05}, define the ``critical points'' $\lambda_1$ and $\lambda_2$, $0<\lambda_1<\lambda_2$, by
\begin{equation}\label{eq:def_c1_c2}
\lambda_1:=\psi'(1)-\psi(1),\;\;\; \lambda_2:=2\psi'(2)-\psi(2).
\end{equation}
We assume that the distribution of the random variable $V$ is lattice. Since $Z_n$ changes only by a constant factor $e^{-na}$ if we replace $V$ by $V-a$, there is no restriction of generality in making the following assumption.
\begin{ass}\label{ass:non_lattice}
There is $h>0$ such that the values of $V$ belong with probability $1$ to the lattice $h\Z=\{hn; n\in\Z\}$, and, moreover, $h$ is the largest number with this property.
\end{ass}
It will be convenient to denote by $[b]_h$ and $\{b\}_h$ the entire part and the fractional part of $b\in\R$ taken with respect to the lattice $h\Z$, i.e.
$$
[b]_h:=\max\{a\in h\Z: a\leq b\},\;\;\; \{b\}_h:=b-[b]_h\in [0,h).
$$
Note that $[b]_1$ and $\{b\}_1$  are the usual integer and fractional parts of $b$.
The next theorem is our main result.
\begin{theorem}\label{theo:stable_lattice}
Suppose that~\eqref{eq:asympt_N}, \eqref{eq:def_varphi} and Assumption~\ref{ass:non_lattice} are satisfied.
Define $\alpha\in (0,2)$ as the unique solution of the equation $\alpha\psi'(\alpha)-\psi(\alpha)=\lambda$.
Define $A_n$ and $B_n$ by
\begin{align}
A_n&=
\begin{cases}
0,     &\textrm{ if }\lambda\in (0,\lambda_1),\\
N_n \E[e^{\sum_{j=1}^n V_{1,j}}1_{\sum_{j=1}^n V_{1,j}< b_n}], &\textrm{ if }    \lambda=\lambda_1,\\
\E Z_n, & \textrm{ if }\lambda\in (\lambda_1,\lambda_2),
\end{cases}\label{eq:def_an_lattice}\\
B_n&=e^{b_n}, \text{ where }
b_n=nI^{-1}\left(\frac1n\log \left(\frac{ N_nh}{\sqrt{2\pi\psi''(\alpha) n}}\right)\right).\label{eq:def_bn_lattice}
\end{align}
If $\{n_k\}_{k\in\N}$ is an increasing integer sequence such that
\begin{equation}\label{eq:def_Delta}
\Delta:=\lim_{k\to\infty}\{b_{n_k}\}_h\in [0,h],
\end{equation}
then we have the following weak convergence:
\begin{equation}\label{eq:main_conv}
\frac{Z_{n_k}-A_{n_k}}{B_{n_k}}\todistr \FF_{\alpha, \Delta}, \;\;\; k\to\infty.
\end{equation}
Here, $\FF_{\alpha, \Delta}$ is an infinitely divisible distribution whose characteristic function $\phi_{\alpha,\Delta}$ has a L\'evy--Khintchine representation
\begin{equation}\label{eq:char_func}
\log \phi_{\alpha,\Delta}(u)=
iC_{\alpha,\Delta;\tau}u + \sum_{x\in  e^{h\Z-\Delta}}\left(e^{iux}-1- iux 1_{x<\tau}\right) x^{-\alpha}, \;\;\; u\in\R,
\end{equation}
where $e^{h\Z-\Delta}$ denotes the geometric progression $\{e^{hn-\Delta}; n\in\Z\}$, $\tau>0$ is arbitrary such that $\tau\notin e^{h\Z-\Delta}$, and $C_{\alpha,\Delta;\tau}$ is a constant.
\end{theorem}

\begin{corollary}\label{cor:tight}
Under the assumptions of Theorem~\ref{theo:stable_lattice}, the sequence of random variables
\begin{equation}\label{eq:seq_Zn}
\left \{\frac{Z_n-A_n}{B_n}; n\in\N \right\}
\end{equation}
is relatively compact in the weak topology. The set of the weak accumulation points of the sequence~\eqref{eq:seq_Zn} is $\{\FF_{\alpha,\Delta}; \Delta\in[0,h]\}$.  Endowed with the induced weak topology, this set is homeomorphic to a circle.
\end{corollary}
\begin{example}
Let the variable $V$ take two values $h$ and $0$ with probabilities $p$ and $1-p$, respectively, $p\in(0,1)$. In order to motivate this choice, consider a game in which a player with starting capital $1$  tosses  a coin $n$ times and each time the coin shows heads (which happens with probability $p$), the capital is multiplied by $e^{h}$. If the coin lands tails, the capital remains unchanged. With other words, the gain of the player in such a game is $e^{hk}$ if the coin lands $k$ times heads, $k=0,\ldots, n$. Then, the random variable $Z_n$ may be interpreted as the total gain in $N_n$ independent games. Theorem~\ref{theo:stable_lattice}  provides a complete description of the subsequential limit laws of $Z_n$ as $n,N_n\to\infty$ provided that the growth condition~\eqref{eq:asympt_N} is satisfied. The critical point $\lambda_2$ is given by
\begin{equation}
\lambda_2=\frac{2phe^{2h}}{(1-p)+pe^{2h}}-\log ((1-p)+pe^{2h}).
\end{equation}
It should be stressed that the central limit theorem~\eqref{eq:conv_normal} as well as the limit results in the intermediate regime $\log N_n\approx \lambda_2 n$ (see~\cite[Thm.~1.3]{kabluchko09d}) remain valid in the lattice case. Also, the strong laws for $Z_n$ proved in~\cite{khorunzhiy03}, \cite{cranston_molchanov05}, \cite{kabluchko09b} hold in the lattice case. Thus, it is only the weak convergence result under the growth condition~\eqref{eq:asympt_N} which is affected by the lattice assumption.
\end{example}
\begin{remark}
The distributions $\FF_{\alpha,\Delta}$ are \textit{semi-stable}. Recall that an infinitely divisible distribution on the real line with characteristic function $\phi$ is called \textit{semi-stable} with index $\alpha\in (0,2]$ if for \textit{some} positive $a\neq 1$, there exists $c\in\R$ such that $(\phi(u))^a=e^{ict}\phi(a^{1/\alpha}u)$; see~\cite[Ch.~3]{sato_book}. Stable distributions are obtained by requiring the same condition to hold for \textit{every} $a>0$ (with $c$ depending on $a$). Semi-stable distributions arise as subsequential weak limits of the partial sums of i.i.d.\ random variables taken along geometrically growing subsequences; see, e.g., \cite{pillai71}. This setting is applicable for example to the total gain in a large number of St.\ Petersburg games. Recall that in a St.\ Petersburg game, the gain of a player is $2^k$ with probability $2^{-k}$, $k\in\N$. If $S_N$ denotes the total gain in $N$ independent St.\ Petersburg games, then the random variable $S_N$ does not converge to a limiting distribution as $N\to\infty$. However, it has been observed  by~\citet{martin_lof85} that the \textit{subsequence} $S_{2^N}$ has a limiting distribution as $N\to\infty$. Later, the full picture of limiting semi-stable laws arising as subsequential limits of $S_n$ has been established in~\cite{csorgo_dodunekova91}. It is interesting to note that although the sum of independent products $Z_n$ does not fit in this setting (rather, $Z_n$ is a row sum in a triangular array), the structure of the set of limiting distributions is very similar to that encountered in~\cite{csorgo_dodunekova91}. In particular, the limiting distributions are semi-stable (and not only infinitely divisible, which is clear a priori).
\end{remark}
\begin{remark}
In~\eqref{eq:def_bn_lattice}, we agree to take the values of the inverse function $I^{-1}$ to be in the interval $(\beta_0,\beta_{+\infty})$. Note that by~\eqref{eq:asympt_N}, see also Eqn.~\eqref{eq:I_varphi_prime} below, we have $b_n\sim \psi'(\alpha) n$ as $n\to\infty$.
\end{remark}
\begin{remark}
The value of the constant $C_{\alpha, \Delta; \tau}$ is given by Eqns.~\eqref{eq:C_alpha_less1}, \eqref{eq:C_alpha_eq1}, \eqref{eq:C_alpha_larger1} below for $\alpha\in (0,1)$, $\alpha=1$, $\alpha\in (1,2)$, respectively. It is easily seen from these equations that the right-hand side of~\eqref{eq:char_func} does not depend on the choice of $\tau$.
\end{remark}
\begin{remark}
Our growth condition~\eqref{eq:asympt_N} is less restrictive than the corresponding assumption in~\cite{cranston_molchanov05}, where $N_n$ is chosen to be of the form $N_n=\alpha(2\pi \psi''(\alpha)n)^{1/2}e^{\lambda n}$. See~\cite[Theorem~1.4]{kabluchko09d} for the proof of~\eqref{eq:conv_stable} in the non-lattice case under~\eqref{eq:asympt_N}.
\end{remark}
A quantity closely related to the \textit{sum} of independent products $Z_n$ is the \textit{maximum} of independent products
\begin{equation}
M_n=\max_{i=1,\ldots,N_n} \prod_{j=1}^{n} e^{V_{i,j}}.
\end{equation}
Equivalently, one may consider the \textit{maximum of independent sums} 
$$
\log M_n=\max_{i=1,\ldots,N_n} \sum_{j=1}^{n} V_{i,j}.
$$ 
The limiting behavior of $\log M_n$ as $n\to\infty$ have been studied independently by a number of authors including~\cite{ivchenko73}, \cite{durrett79}, \cite{lifshits}, \cite{cranston_molchanov05}. Also, two versions of $\log M_n$ with an additional dependence between the sums $\sum_{j=1}^{n} V_{i,j}$, $i=1,\ldots,N_n$, have been studied in~\cite{komlos_tusnady75} (in the context of the Erd\"os--Renyi law of large numbers), and in~\cite[Chapter~3]{bollobas_book} (the maximal degree of a vertex in a random graph). Another related model is the maximum of the branching random walk; see, e.g., \cite{lifshits}. If the random variable $V$ is non-lattice with finite  exponential moments, then $\log M_n$ has limiting Gumbel extreme-value distribution function $e^{-e^{-x}}$. In the lattice case, the convergence to the Gumbel limit breaks down and instead, a family of discrete analogues of the Gumbel distribution appears as the set of the weak accumulation points.

The rest of the paper is devoted to the proof of Theorem~\ref{theo:stable_lattice}. Our approach  (which follows the idea used in~\cite{ben_arous_etal05} and~\cite{cranston_molchanov05}) is to view $Z_n$ as a row sum in a triangular array with independent rows and to apply the classical theory of convergence to infinitely divisible distributions. This results in a number of conditions on the truncated exponential moments which need to be verified. The verification is done using the precise large deviation theorems due to~\citet{bahadur_rao60} and~\citet{petrov65}.

\section{Proof of Theorem~\ref{theo:stable_lattice} and Corollary~\ref{cor:tight}}\label{sec:proof_stable_lat}

\subsection{Method of the proof}
Recall that $\{V_{i,j}; (i,j)\in \N^2\}$ and $\{V_j; j\in\N\}$ are independent copies of a random variable $V$ satisfying~\eqref{eq:def_varphi}. For every $n\in\N$, let $W_{1,n},\ldots, W_{N_n,n}$ and $W_n$ be i.i.d.\ random variables  defined by
\begin{equation}\label{eq:def_Wn}
W_{i,n}=e^{\sum_{j=1}^{n} V_{i,j}-b_n},\;\;\; i=1,\ldots, N_n; \;\;\; W_n=e^{\sum_{j=1}^n V_j-b_n}.
\end{equation}
With this notation, Eqn.~\eqref{eq:main_conv} of Theorem~\ref{theo:stable_lattice} is equivalent to the following statement:
\begin{equation}\label{eq:stable_latt_cond}
\sum_{i=1}^{N_{n_k}} W_{i,{n_k}}- B_{n_k}^{-1}A_{n_k}\todistr \FF_{\alpha, \Delta},\;\;\; k\to\infty.
\end{equation}
Note that $\{W_{i,n}; n\in\N, i=1,\ldots,N_n\}$ is a triangular array of positive-valued random variables and that the variables within the same row are independent of each other. By the standard theory of convergence to infinitely divisible distributions (see, e.g., Theorem~1 on page~116 in~\cite{gnedenko_book54}), the convergence in~\eqref{eq:stable_latt_cond} will be established once we have verified the validity of the following three statements:
\begin{enumerate}
\item For every  $\tau>0$ with $\tau \notin e^{h\Z-\Delta}$,
\begin{equation}\label{eq:stable_latt_cond1}
\lim_{k\to\infty} N_{n_k} \P[W_{n_k}>\tau]
=
\sum_{x\in e^{h\Z-\Delta}} x^{-\alpha}1_{x>\tau}.
\end{equation}
\item We have
\begin{equation}\label{eq:stable_latt_cond2}
\lim_{\tau\downarrow 0}\limsup_{k\to\infty} N_{n_k} \Var [W_{n_k} 1_{W_{n_k}\leq \tau}]=0.
\end{equation}
\item  For every $\tau>0$ with $\tau\notin e^{h\Z-\Delta}$, the following limit exists and is finite:
\begin{equation}\label{eq:stable_latt_cond3}
C_{\alpha,\Delta;\tau}:=\lim_{k\to\infty} (N_{n_k}\E[W_{n_k} 1_{W_{n_k}\leq \tau}]-B_{n_k}^{-1}A_{n_k}).
\end{equation}
\end{enumerate}
Note that the first condition identifies the L\'evy measure of the limiting distribution $\FF_{\alpha,\Delta}$, the second condition shows that there is no Gaussian part in the limit, and the last condition identifies the shift parameter in the L\'evy--Khintchine formula. The formula~\eqref{eq:char_func} for the characteristic function of the limiting distribution $\FF_{\alpha,\Delta}$ follows from Eqn.~(8) on page~84 of~\cite{gnedenko_book54}.

\subsection{Facts about large deviations}\label{sec:ld}
We collect some facts on large deviations for sums of independent random variables needed in the sequel. Let $\{V_i; i\in\N\}$ be i.i.d.\ copies of a random variable $V$ satisfying~\eqref{eq:def_varphi}, and let $S_n=V_1+\ldots+V_n$ be their partial sums.
Recall that
\begin{equation}\label{eq:def_varphi1}
\psi(t)=\log \E e^{tV},\;\;\; t\in\R,
\;\;\;
I(\beta)=\sup_{t\in \R} (\beta t-\psi(t)), \;\;\; \beta\in\R.
\end{equation}
Note that $\psi$ is infinitely differentiable, strictly convex, and $\psi(0)=0$. The function $I$ is finite, strictly convex and infinitely differentiable on the interval $(\beta_{-\infty}, \beta_{+\infty})$, and its unique zero is $\beta_0$, where
\begin{equation}\label{eq:def_beta_0}
\beta_{-\infty}:=\lim_{t\to -\infty} \psi'(t),\;\;\;
\beta_0:=\psi'(0)=\E V,\;\;\;
\beta_{+\infty}:=\lim_{t\to +\infty}\psi'(t).
\end{equation}
If $\beta=\psi'(\alpha)$ for some $\alpha\in\R$, then the supremum in~\eqref{eq:def_varphi1} is attained at $t=\alpha$ and hence,
\begin{equation}\label{eq:I_varphi_prime}
I(\psi'(\alpha))=\alpha\psi'(\alpha)-\psi(\alpha),\;\;\; \alpha\in \R.
\end{equation}
The next lemma is standard; see, e.g., \cite[Lemma~3]{kabluchko09b} for the proof.
\begin{lemma}\label{lem:I_prime}
For every $\alpha\in\R$, we have $I'(\psi'(\alpha))=\alpha$.
\end{lemma}

The following theorem on the precise asymptotic behavior of large deviation probabilities for sums of i.i.d.\ variables of~\citet{bahadur_rao60}, \citet{petrov65} (see Theorem~6 therein) will play a crucial role in the sequel. It is this theorem where the difference between the lattice and the non-lattice case comes into play.

\begin{theorem}\label{theo:ld}
Suppose that~\eqref{eq:def_varphi} is satisfied.
For $\beta\in(\beta_{-\infty},\beta_{+\infty})$ define $\alpha$ to be the unique solution of the equation $\psi'(\alpha)=\beta$.
Assume that the distribution of $V$ is lattice, and that Assumption~\ref{ass:non_lattice} is fulfilled for some  $h>0$. 
\begin{enumerate}
\item \label{p:ld_1}
  For every  $\beta\in n^{-1}h\Z$,
\begin{equation}
\P[S_n = n\beta]\sim \frac{he^{-n I(\beta)}}{\sqrt{2\pi \psi''(\alpha)n}},
\;\;\; n\to\infty.
\end{equation}
\item \label{p:ld_2} For every compact set $K\subset (\beta_{0},\beta_{+\infty})$, the following holds uniformly in $\beta\in n^{-1}h\Z\cap K$:
\begin{equation}
\P[S_n \geq n\beta]\sim \frac{he^{-n I(\beta)}}{(1-e^{-\alpha h})\sqrt{2\pi \psi''(\alpha)n}},\;\;\;n\to\infty.\\
\end{equation}
\item \label{p:ld_3}
For every compact set $K\subset (\beta_{-\infty},\beta_{0})$, the following holds uniformly in $\beta\in n^{-1}h\Z\cap K$:
\begin{equation}
\P[S_n \leq n\beta]\sim \frac{he^{-n I(\beta)}}{(1-e^{\alpha h})\sqrt{2\pi \psi''(\alpha)n}},\;\;\; n\to\infty.\\
\end{equation}
\end{enumerate}
\end{theorem}

In our proofs, we will several times use an exponential change of measure. Given $t_0\in\R$, we define $\tilde V$ (dependent on $t_0$) to be a random variable with density
\begin{equation}\label{eq:def_tilde_V}
\P[\tilde V=dx]=e^{t_0 x-\psi(t_0)}\P[V=dx].
\end{equation}
Note that the right-hand side is a probability measure since $\E[e^{t_0 V- \psi(t_0)}]=1$.
\begin{lemma}\label{lem:tilde_I}
The Laplace transform $\tilde \psi$ and the information function $\tilde I$ corresponding to $\tilde V$ are given by
\begin{align}
\tilde \psi(t)&=\psi(t+t_0)-\psi(t_0),\;\;\;t\in\R, \label{eq:tilde_psi}\\
\tilde I(\beta)&=I(\beta)+\psi(t_0)-t_0 \beta, \;\;\; \beta\in(\beta_{-\infty},\beta_{+\infty}). \label{eq:tilde_I}
\end{align}
\end{lemma}
\begin{proof}
The formula for $\tilde \psi$ follows immediately from~\eqref{eq:def_tilde_V}. To prove the formula for $\tilde I$, note that by~\eqref{eq:tilde_psi},
$$
\tilde I(\beta)=\sup_{t\in\R} (\beta t-\tilde \psi(t))=\sup_{t\in\R}(\beta(t+t_0)-\psi(t+t_0))+\psi(t_0)-t_0\beta.
$$
Since the supremum on the right-hand side equals $I(\beta)$, Eqn.~\eqref{eq:tilde_I} follows.
\end{proof}

Let $\{\tilde V_i; i\in\N\}$ be independent copies of $\tilde V$ and denote by $\tilde S_n=\tilde V_1+\ldots+\tilde V_n$ their partial sums. By computing the Laplace transforms one obtains immediately that
\begin{equation}\label{eq:def_tilde_Sn}
\P[\tilde S_n=dx]=e^{t_0 x-\psi(t_0)n}\P[S_n=dx].
\end{equation}

\subsection{An auxiliary lemma}
\begin{lemma}\label{lem:aux}
Let the assumptions of Theorem~\ref{theo:stable_lattice} be fulfilled and let $\{x_n\}_{n\in\N}$ be a sequence with $\lim_{n\to\infty}x_n=x$. Then,
\begin{equation}\label{eq:tech2}
nI\left(\frac{b_n+x_n}{n}\right)
=\log \left(\frac{ N_nh}{\sqrt{2\pi\psi''(\alpha) n}}\right)+\alpha x+o(1),\;\;\; n\to\infty.
\end{equation}
\end{lemma}
\begin{proof}
Recall from~\eqref{eq:def_bn_lattice} that $b_n=nI^{-1}(c_n)$, where
\begin{equation}
c_n=\frac1n\log \left(\frac{ N_nh}{\sqrt{2\pi\psi''(\alpha) n}}\right).
\end{equation}
It follows from~\eqref{eq:asympt_N} that $\lim_{n\to\infty}c_n=\lambda$ and hence, $\lim_{n\to\infty} I^{-1}(c_n)=\psi'(\alpha)$. [Recall that $I(\psi'(\alpha))=\lambda$ by~\eqref{eq:I_varphi_prime}].
By Taylor's expansion of the function $I$ around the point $I^{-1}(c_n)$, we have
\begin{align}
I\left(\frac{b_n+x_n}{n}\right)
&=I\left(I^{-1}(c_n)+\frac{x_n}{n}\right)\label{eq:tech1}\\
&=c_n+I'(I^{-1}(c_n))\cdot \frac{x_n}{n}+o\left(\frac 1n\right),\;\;\; n\to\infty. \notag
\end{align}
By the continuity of $I'$ and Lemma~\ref{lem:I_prime}, we have
$$
\lim_{n\to\infty}I'(I^{-1}(c_n))x_n=I'(\psi'(\alpha))x=\alpha x.
$$
Inserting this into~\eqref{eq:tech1} completes the proof of the lemma.
\end{proof}

\subsection{Proof of~\eqref{eq:stable_latt_cond1}}
Recalling that $W_n=e^{S_n-b_n}$  and using the fact that $S_n$ takes values in $h\Z$, we have
$$
\P[W_n>\tau]=\P[S_n>b_n+\log \tau]=\P[S_n> [b_n+\log \tau]_h].
$$
Note that by~\eqref{eq:def_bn_lattice} and~\eqref{eq:asympt_N},
$$
\lim_{n\to\infty} \frac 1n [b_n+\log \tau]_h=I^{-1}(\lambda)=\psi'(\alpha)>\psi'(0)=\E V.
$$
By Theorem~\ref{theo:ld}, Parts~\ref{p:ld_1} and~\ref{p:ld_2},
\begin{align}
\lefteqn{\P[W_n>\tau]}\label{eq:P_Sn_bn_latt}\\
&\sim \frac{he^{-\alpha h}}{(1-e^{-\alpha h}) \sqrt {2\pi \psi''(\alpha) n}} \exp\left\{-n I\left(\frac{[ b_n+\log \tau]_h}{n}\right)\right\}, \;\;\; n\to\infty. \notag
\end{align}
By the assumption~\eqref{eq:def_Delta} of Theorem~\ref{theo:stable_lattice}, $\lim_{k\to\infty}\{b_{n_k}\}_h=\Delta$. Recall also that $\log \tau\notin h\Z-\Delta$. Thus,
$$
\Theta_{\Delta;\tau}:=\lim_{k\to\infty} [b_{n_k}+\log \tau]_h-b_{n_k}=[\Delta+\log \tau ]_h-\Delta.
$$
Restricting Lemma~\ref{lem:aux} to the subsequence $\{n_k\}_{k\in\N}$, we obtain
\begin{equation}\label{eq:i_beta_n_latt}
n_kI\left(\frac{[ b_{n_k}+\log \tau]_h}{n_k}\right)
=\log \left(\frac{N_{n_k}h}{\sqrt{2\pi\psi''(\alpha) n_k}}\right)+\alpha \Theta_{\Delta;\tau} +o(1),\;\;\; k\to\infty.
\end{equation}
Applying~\eqref{eq:i_beta_n_latt} to the right-hand side of~\eqref{eq:P_Sn_bn_latt}, we obtain
\begin{equation}\label{eq:wspom17a}
\lim_{k\to\infty}N_{n_k}\P[W_{n_k}>\tau]=\frac{e^{-\alpha ( \Theta_{\Delta;\tau}+h)}}{1-e^{-\alpha h}}.
\end{equation}
To see that~\eqref{eq:wspom17a} is equivalent to~\eqref{eq:stable_latt_cond1}, note that
$$
\sum_{x\in e^{h\Z-\Delta}} x^{-\alpha}1_{x>\tau}
=
\sum_{k=h^{-1}\cdot [\Delta+\log \tau]_h+1}^{\infty} e^{-\alpha(hk-\Delta)}
=
\frac{e^{-\alpha(\Theta_{\Delta;\tau}+h)}}{1-e^{-\alpha h}}.
$$

\subsection{Proof of~\eqref{eq:stable_latt_cond2}}
Since the variance of a random variable is not greater than the second moment, it suffices to show that
\begin{equation}\label{eq:111}
\lim_{\tau \downarrow 0} \limsup_{n\to\infty} N_{n}\E[W_{n}^2 1_{W_{n}\leq \tau}]=0.
\end{equation}
To estimate the truncated moment $\E[W_{n}^2 1_{W_{n}\leq \tau}]$, we will use an exponential change of measure argument.
Let $\tilde V$ and $\tilde S_n$ be defined as in~\eqref{eq:def_tilde_V} and~\eqref{eq:def_tilde_Sn} with $t_0=2$. By~\eqref{eq:def_Wn} and~\eqref{eq:def_tilde_Sn}, we have
\begin{align}
N_n \E[W_{n}^2 1_{W_{n}\leq \tau}]
&=N_n e^{-2b_n} \E[e^{2S_n}1_{S_n\leq b_n+\log\tau}]\notag\\
&=N_n e^{\psi(2)n} e^{-2b_n} \P[\tilde S_n\leq b_n+\log\tau]\label{eq:wspom_trunc1}\\
&=N_n e^{\psi(2)n} e^{-2b_n} \P[\tilde S_n  \leq [ b_n+\log\tau]_h].\notag
\end{align}
Note that by~\eqref{eq:def_bn_lattice} and~\eqref{eq:asympt_N}, we have
$$
\lim_{n\to\infty}\frac 1n [ b_n+\log\tau]_h=I^{-1}(\lambda)= \psi'(\alpha)<\psi'(2)=\E \tilde V.
$$
Let $\tau<1$ be fixed and denote by $C_1,C_2,\ldots$ constants not depending on $\tau$. By Part~\ref{p:ld_3} of Theorem~\ref{theo:ld},
\begin{equation}\label{eq:wspom_trunc2}
\P[\tilde S_n\leq [ b_n+\log\tau]_h]\sim \frac{C_1}{\sqrt{n}}\exp\left\{-n \tilde I\left(\frac{[ b_n+\log \tau]_h}{n}\right)\right\},\;\;\; n\to\infty,
\end{equation}
where $\tilde I$ is the information function corresponding to $\tilde S_n$. Let $\eps\in (0,2-\alpha)$. Note that by Lemma~\ref{lem:I_prime}, $\lim_{n\to\infty}I'(b_n/n)=I'(\psi'(\alpha))=\alpha$. By the convexity of $I$, we have for sufficiently large $n$,
\begin{equation}\label{eq:444}
nI\left(\frac{b_n+\log\tau}{n}\right)
\geq
nI\left(\frac{b_n}{n}\right)+I'\left(\frac{b_n}{n}\right)\log \tau
\geq
nI\left(\frac{b_n}{n}\right)+(\alpha+\eps)\log\tau.
\end{equation}
Note that $\tilde I$ is decreasing on $(\beta_{-\infty},\psi'(2))$. By Lemma~\ref{lem:tilde_I} and Eqns.~\eqref{eq:444}, \eqref{eq:def_bn_lattice}, we have
\begin{align}
n \tilde I\left(\frac{[ b_n+\log \tau]_h}{n}\right)
&\geq n \tilde I\left(\frac{b_n+\log\tau}{n}\right)\notag\\
&= nI\left(\frac{b_n+\log \tau}{n}\right)+\psi(2)n-2(b_n+\log\tau)\label{eq:wspom_trunc3}\\
&\geq \log \left(\frac{N_n}{\sqrt{n}}\right)+(\alpha+\eps-2)\log\tau+\psi(2)n-2b_n-C_2.\notag
\end{align}
Bringing~\eqref{eq:wspom_trunc1}, \eqref{eq:wspom_trunc2}, \eqref{eq:wspom_trunc3} together, we obtain
$$
\limsup_{n\to\infty}N_n \E[W_{n}^2 1_{W_{n}\leq \tau}]\leq C_3\tau^{2-\alpha-\eps}.
$$
Letting $\tau\downarrow 0$ and recalling that $\alpha+\eps\leq 2$ yields~\eqref{eq:111}.

\subsection{Proof of~\eqref{eq:stable_latt_cond3}}
Let $\tilde V$ and $\tilde S_n$ be the exponential twists of $V$ and $S_n$ defined as in~\eqref{eq:def_tilde_V} and~\eqref{eq:def_tilde_Sn} with $t_0=1$, i.e.,
\begin{equation}\label{eq:def_tilde_Sn1}
\P[\tilde V=dx]=e^{x-\psi(1)}\P[V=dx],
\;\;\;
\P[\tilde S_n=dx]=e^{x-\psi(1)n}\P[S_n=dx].
\end{equation}
It follows from~\eqref{eq:def_Delta} that
\begin{equation}\label{eq:def_xn}
\Theta_{\Delta;\tau}:=\lim_{k\to\infty}\theta_{n_k}=[\Delta+\log \tau ]_h-\Delta, \text{ where } \theta_n=[ b_n+\log\tau]_h-b_n.
\end{equation}

Consider first the case $\alpha\in(0,1)$. Note that in this case, $A_n=0$ by~\eqref{eq:def_an_lattice}. By~\eqref{eq:def_tilde_Sn1}, we have
\begin{align}
N_n \E[W_{n} 1_{W_{n}\leq \tau}]-B_n^{-1}A_n
&=N_n e^{-b_n}\E[e^{S_n}1_{S_n\leq b_n+\log\tau}]\notag\\
&=N_n e^{\psi(1)n} e^{-b_n} \P[\tilde S_n\leq b_n+\log\tau]\label{eq:wspom_tr_c1_1}\\
&=N_n e^{\psi(1)n} e^{-b_n} \P[\tilde S_n\leq [ b_n+\log\tau]_h].\notag
\end{align}
By~\eqref{eq:def_bn_lattice}, \eqref{eq:asympt_N}, and the assumption $\alpha\in (0,1)$,
$$
\lim_{n\to\infty}\frac 1n [ b_n+\log\tau]_h=I^{-1}(\lambda)=\psi'(\alpha)<\psi'(1)=\E\tilde V.
$$
Note that by Lemma~\ref{lem:tilde_I}, $\psi'(\alpha)=\tilde \psi'(\alpha-1)$. By Part~\ref{p:ld_3} of Theorem~\ref{theo:ld},
\begin{align}
\lefteqn{\P[\tilde S_n\leq [ b_n+\log \tau]_h]}\label{eq:wspom_tr_c1_2}\\
&\sim \frac{h} {(1-e^{(\alpha-1) h})\sqrt{2\pi \psi''(\alpha)n}} \exp \left\{-n \tilde I\left(\frac{[ b_n+\log\tau]_h}{n}\right)\right\},\;\;\; n\to\infty.\notag
\end{align}
By~\eqref{eq:def_xn}, Lemma~\ref{lem:tilde_I}, and Lemma~\ref{lem:aux}, we have
\begin{align}
\lefteqn{n_k\tilde I\left(\frac{[ b_{n_k}+\log\tau]_h}{n_k}\right)}\notag\\
&=n_k\tilde I\left(\frac{b_{n_k}+\theta_{n_k}}{n_k}\right)\notag\\
&=n_kI\left(\frac{b_{n_k}+\theta_{n_k}}{n_k}\right)+\psi(1)n_k-(b_{n_k}+\theta_{n_k})\label{eq:wspom_tr_c1_3}\\
&=\log\left(\frac{N_{n_k}h}{\sqrt{2\pi\psi''(\alpha) n_k}}\right)+(\alpha-1) \Theta_{\Delta;\tau}+\psi(1)n_k-b_{n_k}+o(1),\;\;\; k\to\infty.\notag
\end{align}
Bringing~\eqref{eq:wspom_tr_c1_1}, \eqref{eq:wspom_tr_c1_2}, \eqref{eq:wspom_tr_c1_3} together, we obtain
\begin{equation}\label{eq:C_alpha_less1}
C_{\alpha,\Delta;\tau}:=\lim_{k\to\infty}N_{n_k} \E[W_{n_k} 1_{W_{n_k}\leq \tau}]=\frac{e^{-(\alpha-1)\Theta_{\Delta;\tau}}}{1-e^{(\alpha-1) h}}.
\end{equation}

Let us consider the case $\alpha=1$.  We have $B_n^{-1}A_n=N_n\E[W_{n} 1_{W_{n}<1}]$ by~\eqref{eq:def_an_lattice}, \eqref{eq:def_bn_lattice}. Assume for concreteness that $\tau>1$. It follows from~\eqref{eq:def_tilde_Sn1} that
\begin{align*}
N_n \E[W_{n} 1_{W_{n}\leq \tau}]-B_{n}^{-1}A_{n}
&=N_n \E[W_{n} 1_{W_{n}\in [1,\tau]}]\\
&=N_n e^{-b_n}\E[e^{S_n}1_{b_n\leq S_n\leq b_n+\log\tau}]\\
&=N_n e^{\psi(1)n} e^{-b_n} \P[b_n\leq \tilde S_n\leq b_n+\log\tau].
\end{align*}
This may be written as
\begin{equation}\label{eq:wspom_tr_c2_1}
N_n \E[W_{n} 1_{W_{n}\leq \tau}]-B_{n}^{-1}A_{n}=
N_n e^{\psi(1)n} e^{-b_n}\sum_{\substack{0\leq j\leq \log\tau\\ j\in h\Z-b_n}}
\P[\tilde S_n=b_n+j].
\end{equation}
By Part~\ref{p:ld_1} of Theorem~\ref{theo:ld}, we have
\begin{align}
\P[\tilde S_n=b_n+j]
\sim
\frac{h} {\sqrt{2\pi \psi''(1)n}} \exp \left\{-n \tilde I\left(\frac{b_n+j}{n}\right)\right\},\;\;\; n\to\infty. \label{eq:wspom_tr_c2_2}
\end{align}
By Lemma~\ref{lem:tilde_I} and Lemma~\ref{lem:aux},
\begin{align}
\lefteqn{n_k \tilde I\left(\frac{b_{n_k}+j}{n_k}\right)}\notag\\
&=
n_k I\left(\frac{b_{n_k}+j}{n_k}\right)+\psi(1)n_k-(b_{n_k}+j)\label{eq:wspom_tr_c2_3}\\
&=
\log\left(\frac{N_{n_k}h}{\sqrt{2\pi\psi''(\alpha) n_k}}\right)+\psi(1)n_k-b_{n_k}+o(1),\;\;\; k\to\infty.\notag
\end{align}
Note that the right-hand side does not depend on $j$. It follows from~\eqref{eq:def_Delta} that for sufficiently large $n$, the number of summands on the right-hand side of~\eqref{eq:wspom_tr_c2_1} is equal to $h^{-1}\cdot[\log\tau+\Delta]_h$.   Using~\eqref{eq:wspom_tr_c2_1}, \eqref{eq:wspom_tr_c2_2}, \eqref{eq:wspom_tr_c2_3}, we obtain
\begin{equation}\label{eq:C_alpha_eq1}
C_{1,\Delta; \tau}:=\lim_{k\to\infty}
(N_{n_k}\E[W_{n_k} 1_{W_{n_k}\leq \tau}]-B_{n_k}^{-1}A_{n_k})
=
\frac 1h \cdot [\log\tau+\Delta]_h. 
\end{equation}

Finally, let us consider the case $\alpha\in (1,2)$. First note that by~\eqref{eq:def_an_lattice}, \eqref{eq:def_bn_lattice}, we have $B_n^{-1}A_n=N_n\E[W_n]$.
By~\eqref{eq:def_tilde_Sn1},
\begin{align}
N_n \E[W_{n} 1_{W_{n}\leq \tau}]-B_{n}^{-1}A_{n}
&=-N_n e^{-b_n}\E[e^{S_n}1_{S_n>b_n+\log\tau}]\notag\\
&=-N_n e^{\psi(1)n} e^{-b_n} \P[\tilde S_n>b_n+\log\tau]\label{eq:wspom_tr_c3_1}\\
&=-N_n e^{\psi(1)n} e^{-b_n} \P[\tilde S_n>[b_n+\log\tau]_h].\notag
\end{align}
Note that by~\eqref{eq:def_bn_lattice}, \eqref{eq:asympt_N}, and the assumption $\alpha\in (1,2)$, we have
$$
\lim_{n\to\infty}\frac 1n [ b_n+\log\tau]_h=I^{-1}(\lambda)=\psi'(\alpha)>\psi'(1)=\E\tilde V.
$$
By Lemma~\ref{lem:tilde_I}, $\psi'(\alpha)=\tilde \psi'(\alpha-1)$. By Parts~\ref{p:ld_1} and~\ref{p:ld_2} of Theorem~\ref{theo:ld},
\begin{align}
\lefteqn{\P[\tilde S_n> [b_n+\log \tau]_h]}\label{eq:wspom_tr_c3_2}\\
&\sim \frac{h} {(e^{(\alpha-1) h}-1) \sqrt{2\pi \psi''(\alpha)n}} \exp \left\{-n \tilde I\left(\frac{[ b_n+\log\tau]_h}{n}\right)\right\},\;\;\; n\to\infty. \notag
\end{align}
Recall that $\Theta_{\Delta;\tau}$ and $\theta_n$ are given by~\eqref{eq:def_xn}.
By Lemma~\ref{lem:tilde_I} and Lemma~\ref{lem:aux}, we have
\begin{align}
\lefteqn{n_k\tilde I\left(\frac{[b_{n_k}+\log\tau]_h}{n_k}\right)}\notag\\
&=n_k\tilde I\left(\frac{b_{n_k}+\theta_{n_k}}{n_k}\right)\notag\\
&=n_kI\left(\frac{b_{n_k}+\theta_{n_k}}{n_k}\right)+\psi(1)n_k-(b_{n_k}+\theta_{n_k})\label{eq:wspom_tr_c3_3}\\
&=\log\left(\frac{N_{n_k}h}{\sqrt{2\pi\psi''(\alpha) n_k}}\right)+(\alpha-1) \Theta_{\Delta;\tau}+\psi(1)n_k-b_{n_k}+o(1),\; k\to\infty.\notag
\end{align}
Bringing~\eqref{eq:wspom_tr_c3_1}, \eqref{eq:wspom_tr_c3_2}, \eqref{eq:wspom_tr_c3_3} together, we obtain
\begin{equation}\label{eq:C_alpha_larger1}
C_{\alpha,\Delta;\tau}=\lim_{k\to\infty} N_{n_k} \E[W_{n_k} 1_{W_{n_k}\leq \tau}]-B_{n_k}^{-1}A_{n_k}=\frac{e^{-(\alpha-1)\Theta_{\Delta;\tau}}}{1-e^{(\alpha-1) h}}.
\end{equation}
This completes the proof of~\eqref{eq:stable_latt_cond3} and the proof of Theorem~\ref{theo:stable_lattice}.

\subsection{Proof of Corollary~\ref{cor:tight}}
The relative compactness of the sequence~\eqref{eq:seq_Zn} (as well as the description of its weak cluster set) follow from the fact that from every increasing integer sequence we can extract a subsequence $n_k$ satisfying~\eqref{eq:def_Delta} with some $\Delta\in [0,h]$ and then apply Theorem~\ref{theo:stable_lattice}.

Let us prove that for every fixed $\alpha\in (0,2)$, the set $\{\FF_{\alpha,\Delta}; \Delta\in [0,h]\}$ is homeomorphic to a circle.  Recall from~\eqref{eq:char_func} that the logarithm of the characteristic function of  $\FF_{\alpha, \Delta}$ is given by
\begin{equation}\label{eq:char_func1}
\log \phi_{\alpha,\Delta}(u)=
iC_{\alpha,\Delta;\tau}u + \sum_{x\in  e^{h\Z-\Delta}}\left(e^{iux}-1- iux 1_{x<\tau}\right) x^{-\alpha},\;\;\; u\in\R,
\end{equation}
where $\tau>0$ is arbitrary with $\tau\notin e^{h\Z-\Delta}$.
It follows from~\eqref{eq:def_xn} that $\Theta_{0;\tau}=\Theta_{h;\tau}$. Then, Eqns.~\eqref{eq:C_alpha_less1}, \eqref{eq:C_alpha_eq1}, \eqref{eq:C_alpha_larger1} imply that $C_{\alpha,0;\tau}=C_{\alpha,h;\tau}$. Trivially, we have $e^{h\Z}=e^{h\Z-h}$. By~\eqref{eq:char_func1}, it follows from these facts that $\FF_{\alpha,0}=\FF_{\alpha,h}$. On the other hand, the L\'evy measure of $\FF_{\alpha,\Delta}$ is given by $\sum_{x\in  e^{h\Z-\Delta}} x^{-\alpha}\delta_x$, which implies that the distributions $\FF_{\alpha,\Delta}$, $\Delta\in [0,h)$, are different.

To complete the proof, we need to show that $\FF_{\alpha,\Delta}$ depends continuously (in the weak topology) on $\Delta$. Take some $\Delta_0\in [0,h]$ and choose $\tau>0$ such that $\tau\notin e^{h\Z-\Delta_0}$.  It is easily seen from~\eqref{eq:def_xn} and~\eqref{eq:C_alpha_less1}, \eqref{eq:C_alpha_eq1}, \eqref{eq:C_alpha_larger1} that $C_{\alpha, \Delta;\tau}$ is a continuous function of $\Delta$ in a neighborhood of $\Delta_0$. It follows from~\eqref{eq:char_func1} that $\lim_{\Delta\to\Delta_0} \phi_{\alpha,\Delta}(u)=\phi_{\alpha,\Delta_0}(u)$ for every $u\in\R$. By the L\'evy continuity theorem, this implies that $\FF_{\alpha,\Delta}$ is a continuous function of $\Delta$. This completes the proof of Corollary~\ref{cor:tight}.


\section*{Acknowledgements} The author is grateful to Leonid Bogachev for pointing out the references~\cite{cranston_molchanov05} and~\cite{khorunzhiy03}.

\bibliographystyle{plainnat}
\bibliography{paper15Cbib}
\end{document}